\documentclass[11pt]{amsart}
\usepackage{amsmath, amssymb, amsfonts, verbatim, amsthm, xcolor}
\usepackage{graphics, epsfig, enumerate, amstext,  comment}
\usepackage{graphicx, ulem, hyperref}
\usepackage[T1]{fontenc}
\newtheorem{theorem}{Theorem}[section]

\newtheorem{proposition}[theorem]{Proposition}
\newtheorem{corollary}[theorem]{Corollary}
\newtheorem{lemma}[theorem]{Lemma}
\newtheorem{example}[theorem]{Example}

\def\qed{\hfill $\Box$\medskip}
\def\diag{{\rm diag}\,}

\def\Span{\mathop{\rm Span}\,}

\def\Re{{\rm Re}}

\def\IR{{\mathbb{R}}}
\def\PP{\mathbb{P}}
\def\RR{\mathbb{R}}

\def\KK{{\mathbb{K}}}

\def\HH{{\mathbb{H}}}

\def\FF{{\mathbb F}}
\def\CC{{\mathbb C}}
\def\Span{\mathop{\mathrm{Span}}}

\def\RR{{\mathbb R}}

\begin{document}
\openup 1\jot

\title{Non-linear classification of finite-dimensional simple $C^*$-algebras}

\author[B. Kuzma]{Bojan Kuzma$^{1}$}
\address{$^{a}$University of Primorska, Glagolja\v{s}ka 8, SI-6000 Koper, Slovenia, and
$^{b}$Institute of Mathematics, Physics, and Mechanics, Jadranska 19, SI-1000 Ljubljana, Slovenia.}
 \email{bojan.kuzma@upr.si}
\thanks{$^{1,2}$This work is supported in part by the Slovenian Research Agency (research program P1-0285 and research projects N1-0210, N1-0296 and J1-50000).}

\author[S. Singla]{Sushil Singla$^{2}$}
\address{$^{2}$University of Primorska, Glagolja\v{s}ka 8, SI-6000 Koper, Slovenia,}
 \email{ss774@snu.edu.in}


\begin{abstract}
    A Banach space characterization of simple real or complex $C^*$-algebras is given which even characterizes the underlying field. As an application, it is shown that if $\mathfrak A_1$ and $\mathfrak A_2$ are Birkhoff-James isomorphic simple $C^*$-algebras over the fields $\mathbb F_1$ and $\mathbb F_2$, respectively and if $\mathfrak A_1$ is finite-dimensional with dimension greater than one, then $\mathbb F_1=\mathbb F_2$ and $\mathfrak A_1$ and $\mathfrak A_2$ are (isometrically) $\ast$-isomorphic $C^*$-algebras. 
\end{abstract}

\maketitle

\section{Introduction}

Recall that on a given $\ast$-algebra over the underlying field~$\mathbb F\in\{{\mathbb R}, {\mathbb C}\}$  there can exist only one norm that makes it into a $C^*$-algebra: namely,  the norm of any element $x$ in a $C^*$-algebra is the square root of the spectral radius of $x^*x$  (for more details we refer to \cite[Theorem 4.1]{goodearl} for $\mathbb F=\mathbb C$ and \cite[Theorem 11.1]{goodearl} for $\mathbb F=\mathbb R$). 
Thus, knowing the algebraic structure of a $C^*$-algebra determines its norm and hence determines the $C^*$-algebra completely. 

We can reverse the narrative and ask the following question. Does knowing the Banach space structure of a $C^*$-algebra characterizes its algebraic structure and hence the $C^*$-algebra completely? For example, Kadison \cite{kadison} proved that a linear surjective isometry $\Phi$ between two unital complex $C^*$-algebras is nothing but a Jordan $\ast$-isomorphism multiplied by a fixed unitary element,~$\Phi(I)$. 
This shows that the unit sphere of one unital complex $C^\ast$-algebra $\mathfrak B_1$ is linearly mapped onto the unit sphere of another unital complex $C^*$-algebra $\mathfrak B_2$ if and only if $\mathfrak B_1$ and $\mathfrak B_2$ are Jordan $\ast$-isomorphic up to multiplication by a unitary element. Tanaka \cite{Tanaka} recently considered this further and gave a characterization of abelian complex $C^*$-algebras among all complex $C^*$-algebras in terms of \textit{geometric structure} alone and without assorting to the linearity. In fact, for a maximal face $F$ of the unit ball of a complex $C^*$-algebra $\mathfrak B$, the union of all its supporting hyperplanes, transferred to the origin, 
was denoted by $I_F$. Then, the geometric structure of $\mathfrak B$, denoted by $\mathfrak S(\mathfrak B)$, was (equivalently, see~\cite[Theorem 4.15]{tanaka1}) defined  as the collections of all $I_F$ for all maximal faces of the unit ball of $\mathfrak B$, i.e. $$\mathfrak S(\mathfrak B) =\{I_F;\;\; F\text{ is a maximal face of unit ball of } \mathfrak B\}.$$ 
It was proved in \cite[Theorem 3.5]{Tanaka} that a complex $C^*$-algebra $\mathfrak B$ is abelian if and only if the set $\mathfrak C(\mathfrak B) = \{S\subset \mathfrak S(\mathfrak B);\;\; S=S^{=}\}$ satisfies the axioms of closed sets, where $S^{=} = \{I\in \mathfrak S(\mathfrak B);\;\; \cap_{J\in S}J\subset I\}\text{ for all } S\subseteq \mathfrak S(\mathfrak B).$ Moreover, we have been informed by Tanaka that he also obtained a classification of finite-dimensional complex $C^\ast$-algebras in terms of geometric structure~\cite{tanaka2}. The notion of geometric structure space was initially defined in terms of \textit{Birkhoff-James orthogonality} in \cite{tanaka1}. In this article, we will use the notion of Birkhoff-James orthogonality directly to prove our main results. As an application of our results, we will  be able to tell if a complex $C^\ast$ algebra is finite-dimensional  or not and if it is, whether it is simple or not  (see Remark (C) below for more details). 

In \cite{Tanaka, tanaka2}, the author was working with complex $C^*$-algebras. Note that $C^*$-algebras can also be defined over real numbers as the underlying field: a real $C^*$-algebra $\mathfrak A$ is a real Banach algebra with an involution that satisfies the $C^*$-identity ($\|a^*a\|=\|a\|^2$ for all $a$) as well as  the condition that $1+a^*a$ is invertible in its unitization, see \cite{goodearl}. The second condition is equivalent to  the spectrum of $a^*a$ being contained in non-negative real numbers  and is  automatic for complex $C^*$-algebras -- this was first observed by Irving Kaplansky by using results of Fukamiya \cite{Fukamiya}. But the work of Kaplansky was unpublished and his argument was recorded in Joseph A.  Schatz's review of \cite{Fukamiya} in the Mathematics Reviews. However, it is not automatic for real $C^*$-algebras. For example, the real algebra of complex numbers, $\mathbb C$, with standard norm and identity map as involution satisfies all other assumptions of a real $C^*$-algebra, except the fact that $1+i^*i$ is not invertible. This extra condition of invertibility of $1+a^*a$ in unitization is necessary to have an analogue of GNS construction for real $C^*$-algebra which states that up to $*$-isomorphisms all real $C^*$-algebras are real $\ast$-subalgebras of the space of bounded operators on a real Hilbert space (see \cite[Theorem 15.3]{goodearl} for more details). It turns out that  a Banach $\ast$-algebra $\mathfrak A$ is a real $C^*$-algebra if and only if a  complexification $\mathfrak A_c=\mathfrak A+i\mathfrak A=\mathfrak A\otimes_{\IR}\CC$ of $\mathfrak A$ can be equipped with a norm, extending the original norm of $\mathfrak A$, so that $\mathfrak A_c$ is a $C^*$-algebra (see \cite[Theorem 15.4]{goodearl}). Yet another equivalent definition of a real $C^*$-algebra $\mathfrak A$ is $\|a\|^2\leq \|a^*a+b^*b\|$ for all $a,b\in\mathfrak A$, see \cite[Theorem 1]{palmer}, and see \cite{DorenBelfi} 
 or~\cite{Schroder} for even more equivalent definitions.

Our main theorem of this article will give a procedure to classify the algebraic structure of finite-dimensional simple $C^*$-algebras over $\mathbb F$,  using only the  Birkhoff-James orthogonality relation. As a special case of our theorem, we get the following. If $\mathfrak A_1$ and $\mathfrak A_2$ are simple $C^*$-algebras over $\mathbb F_1$ and $\mathbb F_2$ respectively, (nonlinearly) isomorphic  with respect to the structure of Birkhoff-James orthogonality (that is, if there exists a bijection $\Phi$ such that $x$ is Birkhoff-James orthogonal to $y$ if and only if $\Phi(x)$ is Birkhoff-James orthogonal to $\Phi(y)$), and if either $\mathfrak A_1$ or $\mathfrak A_2$ is finite-dimensional with dimension greater than one, then $\mathbb F_1=\mathbb F_2$ and the two $C^*$-algebras are (isometrically) $*$-isomorphic. In Section \ref{section2}, we state our main results and in Section \ref{section3} we prove them. We end our article with some concluding remarks in Section \ref{section4}.

\section{Preliminaries and statements of the main results}\label{section2}

Throughout,  $\mathbb {\mathcal M}_n(\mathbb K)$ will denote  the set  of $n$-by-$n$ matrices over the real finite-dimensional division algebra 
$\mathbb K=\mathbb R$ (reals), $\mathbb K=\mathbb C$ (complexes) or $\mathbb K=\mathbb H$ (quaternions). A matrix from  $\mathbb {\mathcal M}_n(\mathbb K)$ acts on  $\mathbb K^n$, the $n$-dimensional right vector space of column vectors (that is, $n$-by-$1$ matrices) over $\mathbb K$ by left multiplication and induces a $\mathbb K$-linear mapping in ${\mathbb K}^n$ (if $\mathbb K \in\{\mathbb R, \mathbb C\}$, then ${\mathbb K}^n$ is same as left vector space; also, if $\lambda \in{\mathbb R}$, then $x\lambda=\lambda x$ for $x\in{\mathbb K}^n$). For $ v\in\mathbb K^n$, let  $v^*$ denote the row vector obtained by  the conjugate transpose  of $v$. Notice that  $uv^*$  ($u,v\in\mathbb K^n$) is then a matrix of (column) rank at most one. (Remark that quaternionic matrices may have different column and row ranks).  The norm of $x\in\KK^n$ is defined by $\sqrt{x^\ast x}=\sqrt{\Re(x^\ast x)}$; notice that $\langle x,y\rangle:=\Re(x^\ast y)$ is an inner product on $\KK^n$ when considered as an $(dn)$-dimensional real vector space ($d:=\dim_{\RR}\KK$). Then, the norm of a matrix $A$ is  the  operator norm, induced from its action of ${\mathbb K}^n$. This way, $(\mathbb {\mathcal M}_n(\mathbb K),\|\cdot\|)$ becomes a (simple, finite-dimensional) $C^*$-algebra.

We note that $\mathbb {\mathcal M}_n(\mathbb K)$ are the basic building blocks for the finite-dimensional $C^*$-algebras: any complex finite-dimensional $C^*$-algebra $\mathfrak A$ is $\ast$-isomorphic to $\mathbb {\mathcal M}_{n_1}(\mathbb C)\oplus\dots\oplus\mathbb {\mathcal M}_{n_\ell}(\mathbb C)$ for some positive integers $n_1$, \dots, $n_\ell$ (see  \cite[Theorem 1.5]{goodearl}) and any real finite-dimensional $C^*$-algebra $\mathfrak A$ is $\ast$-isomorphic to $\mathbb {\mathcal M}_{n_1}(\mathbb K_1)\oplus\dots\oplus\mathbb {\mathcal M}_{n_\ell}(\mathbb K_\ell)$ where $\mathbb K_i\in\{\mathbb R, \mathbb C, \mathbb H\}$ 
(see  \cite[Theorem 8.4]{goodearl}). 
We will refer to $\mathbb {\mathcal M}_{n_1}(\mathbb K_1)\oplus\dots\oplus\mathbb {\mathcal M}_{n_\ell}(\mathbb K_\ell)$ as the block decomposition of~$\mathfrak A$. By Wedderburn–Artin theorem, this block decomposition is unique up to the permutation of blocks. Moreover, 
all finite-dimensional simple complex $C^*$-algebras are isomorphic to $\mathbb {\mathcal M}_n(\mathbb C)$ and all finite-dimensional simple real $C^*$-algebras are isomorphic to $\mathbb {\mathcal M}_n(\mathbb R)$, $\mathbb {\mathcal M}_n(\mathbb C)$, or  $\mathbb {\mathcal M}_n(\mathbb K)$.

Let $(V,\|\cdot\|)$ be a normed space over the field $\FF$ and let $x,y\in V$. We say that $x$ is Birkhoff-James (shortly BJ) orthogonal to $y$, and denote it by $x\perp y$, if $$\|x\|\leq \|x+\lambda y\|\text{ for all }\lambda\in \mathbb F.$$
Note that $x\perp y$ does not imply $y\perp x$. 
Two normed spaces $V_1$ and $V_2$ are BJ isomorphic if there exists a bijection $\Phi\colon V_1\to V_2$ such that $x\perp y$ if and only if $\Phi(x)\perp \Phi(y)$. The easiest way to study this is by  associating  a directed graph, $\Gamma_0=\Gamma_0(V)$,  (which we call  a \textit{spatial ortho-digraph}) to  every normed space $V$. Its  vertices are all the elements of $V$ and two vertices $x,y$ form a directed edge $x\rightarrow y$ if $x\perp y$. 
Note that $0\perp x$ and $x\perp 0$ for every $x\in V$ while $x\not\perp x$ if $x\neq 0$. So,  $0\in \Gamma_0(V)$ is the only vertex connected to all vertices and also the only vertex with a loop. Moreover, $V_1$ and $V_2$ are BJ isomorphic if and only if $\Gamma_0(V_1)$ and $\Gamma_0(V_2)$ are isomorphic as digraphs. Since BJ orthogonality relation is homogeneous in the sense that $x\perp y$ if and only if $(\lambda x)\perp (\mu y)$ for every $\lambda,\mu\in\FF$, there is naturally another digraph, $\Gamma=\Gamma(V)$ associated with $V$, called the \textit{projective ortho-digraph} (termed also ortho-digraph in our previous paper~\cite{Aram}, where we investigated only projective versions). 
Its vertices are all  one-dimensional subspaces $[x]=\FF x$ for $x\in V\setminus\{0\}$ (that is, all points in a projective space~$\PP(V)$) and two vertices $[x],[y]$ form a directed edge if some 
(hence any) representative $x\in[x]$ is BJ orthogonal to some (hence any) representative  $y\in[y]$.  Unlike $\Gamma_0$  the vertices  of~$\Gamma$ have no loops and $\Gamma$ can be obtained from $\Gamma_0$ by removing its only loop vertex and contracting, in the remaining graph, vertices which correspond to linearly dependent vectors.  We note that for two normed spaces $V_1$ and $V_2$, the isomorphism between $\Gamma(V_1)$ and $\Gamma(V_2)$ induces the isomorphism between $\Gamma_0(V_1)$ and $\Gamma_0(V_2)$ (see \cite[Theorem 2.5]{tanaka3} and \cite[Lemma 3.1]{kuzma1}). The converse, however, is false in general:  the isomorphism between $\Gamma_0(V_1)$ and $\Gamma_0(V_2)$ may not induce the isomorphism between $\Gamma(V_1)$ and $\Gamma(V_2)$ (see \cite[Example 3.16]{tanaka3} and \cite[Examples 3.13, 3.17]{kuzma1}). 

In \cite{tanaka3}, yet another associated directed graph, the {\it reduced ortho-digraph} $\hat{\Gamma}_0(V)$ was defined to study BJ isomorphism between two normed spaces. To define it one first introduces for each vertex  $x\in\Gamma_0(V)$ its \textit{incoming} and its \textit{outgoing neighborhood}, by
$${}^\bot x:=\{z\in\Gamma_0(V);\;\;z\rightarrow x\}\quad\hbox{ and }\quad x^\bot :=\{z\in\Gamma_0(V);\;\;x\rightarrow z\}$$
and then defines 
an equivalence relation on $V$ by declaring that $x\sim y$ if  ${}^\bot x={}^
\bot y$ and  $x^\bot=y^\bot$.
The reduced ortho-digraph $\hat{\Gamma}_0(V)$ of $V$ is then the quotient of $\Gamma_0(V)$, that is, its  vertices are $X/\sim$ and  vertices $[x]$, $[y]$ form a directed edge if some (hence any) representative $x\in[x]$ is BJ orthogonal to some (hence any) representative  $y\in[y]$. For two normed spaces $V_1$ and $V_2$, the isomorphism between $\Gamma_0(V_1)$ and $\Gamma_0(V_2)$ implies the isomorphism between $\hat{\Gamma}_0(V_1)$ and $\hat{\Gamma}_0(V_2)$ (see \cite[Theorem 2.6]{tanaka3}), but it is unknown whether the converse implication holds or not.

We will  use the symbol $\boldsymbol{\Gamma}=\boldsymbol{\Gamma}(V)$ to stand for $\Gamma_0$, $\Gamma$ or  $\hat{\Gamma}_0$. The incoming and outgoing neighborhood of a vertex  $[x]\in \boldsymbol{\Gamma}(V)$ are defined as before for $\Gamma_0(V)$ and it is easily seen that, for  two vertices $[x_1],[x_2]$ in $\Gamma(V)$ or in $\hat{\Gamma}(V)$, respectively, we have
\begin{equation}\label{eq:connect}
    \big([x_1]^\bot\subseteq [x_2]^\bot\big)\quad\text{ if and only if }\quad\big(x_1^\bot\subseteq x_2^\bot\big).
\end{equation}

Let $V_1$ and $V_2$ be two normed spaces over $\mathbb F$. For brevity, we will use the notation $V_1\sim_{\boldsymbol{\Gamma}} V_2$, used also by Tanaka \cite{Tanaka, tanaka1}, to stand for the statement that the digraphs $\boldsymbol{\Gamma}(V_1)$ and $\boldsymbol{\Gamma}(V_2)$ are isomorphic.
As mentioned above, we have the following implications: $$V_1\sim_{\Gamma} V_2\implies V_1\sim_{\Gamma_0} V_2\implies V_1\sim_{\hat{\Gamma}_0} V_2.$$ The converse implications may not hold in general but do hold in case $V_1$ and $V_2$ are smooth or strictly convex (see \cite[Corollary 2.10]{tanaka3}) or more generally when $V_1$ and $V_2$ are BJ-normed spaces (see \cite[Definition 3.2, Lemma 3.4, Lemma 3.7]{kuzma1}). It is  an immediate corollary to  our main result that
 in the case of two simple $C^*$-algebras $\mathfrak A_1$ and $\mathfrak A_2$, the existence of isomorphism between $\Gamma(\mathfrak A_1)$ and $\Gamma(\mathfrak A_2)$, BJ isomorphism between $\mathfrak A_1$ and $\mathfrak A_2$ and isomorphism between $\hat{\Gamma}(\mathfrak A_1)$ and $\hat{\Gamma}(\mathfrak A_2)$ are all equivalent (and  are further equivalent to the existence of $C^*$-isomorphism between $\mathfrak A_1$ and $\mathfrak A_2$).

We define a preorder on $\boldsymbol{\Gamma}(V)$ by $x\precsim y$ if  $x^\bot\subseteq y^\bot$. Recall that a subset of vertices $\mathcal C\subseteq\boldsymbol{\Gamma}$ is a chain (with respect to the above preorder) if every two vertices in $\mathcal C$ are comparable and have distinct outgoing neighborhoods (that is, $x^\bot\subsetneq y^\bot$ or $y^\bot\subsetneq x^\bot$ for every distinct $x,y\in \mathcal C$). A chain is maximal if it cannot be enlarged to a bigger chain. Note that if ${\bf \Gamma}$ contains a looped vertex, i.e., $0$, then this vertex connects to every other vertex ($0$ is    BJ orthogonal to every element)  and as such ends every maximal chain.  We will thus consider only loopless vertices (i.e., nonzero elements) when examining the properties of maximal chains. The cardinality  of  a chain $\mathcal C$ will also be called its length. For a chain $\mathcal C$ in $(\boldsymbol{\Gamma}(V), \precsim)$, we will call the element $x$ and $z$ of $\mathcal C$ to be first and the last element of $\mathcal C$ if $x^\bot\subsetneq y^\bot$ and $y^\bot\subsetneq z^\bot$ for all $y\in \mathcal C$ respectively. In terms of preorder, we will use the notation $(x_1,\dots, x_m)$ for a chain containing elements $\{x_1, \dots, x_m\}$ when we have $x_1^\bot\subsetneq x_2^\bot\subsetneq\dots\subsetneq x_m^\bot$. 




Note that some properties of the normed space $V$ can also be computed from its ortho-digraph(s) $\boldsymbol{\Gamma}(V)$ alone. We say that these properties are graphological. One such example is the dimension of a space $V$ (another example will be given in Lemma \ref{lem:unitary}).  If the clique number of $\boldsymbol{\Gamma}$ is infinite, then dimension of $V$ is infinite. In case the clique number of $\boldsymbol{\Gamma}$ is finite, the dimension of $V$ is equal to 
the integer $n$  which is the
minimal possible 
cardinality of subsets $\Omega\subseteq \boldsymbol{\Gamma}(V)$ such that $\big|\bigcap_{x\in\Omega} x^\bot\big| = 1$ for $\boldsymbol{\Gamma}(V)=\Gamma_0(V)$ (see  \cite[Theorem 1.1]{kuzma1}); and $\big|\bigcap_{x\in\Omega} x^\bot\big| = 0$ for $\boldsymbol{\Gamma}(V)=\Gamma(V)$ (see \cite[Remark 1.2]{kuzma1}). The same formula computes dimension also for $\boldsymbol{\Gamma}(V)=\hat{\Gamma}_0(V)$ ---  
the arguments follows along the line of \cite[Theorem 1.1]{kuzma1}.

With this in mind, let us introduce the dimension, $\dim\boldsymbol{\Gamma}(\mathfrak A)$, of the digraph(s) $\boldsymbol{\Gamma}=\boldsymbol{\Gamma}(V)$ to be the above integer $n$ if $\boldsymbol{\Gamma}$ has a finite clique-number and $\infty$ if $\boldsymbol{\Gamma}$ has an infinite clique number; this matches with $\dim V$.  With this, we have the following procedure to classify the objects in the category of the finite-dimensional simple $C^*$-algebras over $\mathbb F$ using digraph(s) $\boldsymbol{\Gamma}$ only.

\begin{theorem}\label{thm:simple}
    Let $\mathfrak A$ be a finite-dimensional simple $C^*$-algebra over $\mathbb F$. 
    Then, all maximal chains of loopless vertices in $\boldsymbol{\Gamma}(\mathfrak A)\in \{\Gamma_0(\mathfrak A), \Gamma(\mathfrak A), \hat{\Gamma}_0(\mathfrak A)\}$ have the same finite length~$n$.  If   $\dim{\bf \Gamma}({\mathfrak A})>1$, then we have the following.
    \begin{itemize}
        \item[(i)] If  $\dim\boldsymbol{\Gamma}(\mathfrak A)$ is not a perfect square, then $\mathbb F=\mathbb R$ and $\mathfrak A = \mathbb {\mathcal M}_n(\mathbb C)$.
        \item[(ii)] If  $\dim\boldsymbol{\Gamma}(\mathfrak A)$ is divisible by $4$ and $n=\sqrt{\dim \boldsymbol{\Gamma}(\mathfrak A)}/2$, then $\mathbb F=\mathbb R$ and $\mathfrak A = \mathbb {\mathcal M}_n(\mathbb H)$.
        \item[(iii)] If  $\dim\boldsymbol{\Gamma}(\mathfrak A)=n^2$ and $\boldsymbol{\Gamma}(\mathfrak A)$ satisfies the property that for some (hence any) maximal chain of loopless vertices $(A_1,\dots, A_n)$ in $\boldsymbol{\Gamma}(\mathfrak A)$, the cardinality of 
        $$\{X^\bot;\;\;X\in \boldsymbol{\Gamma}(\mathfrak A) \hbox{ is loopless and } A_{n-1}^\bot\subsetneq X^\bot\}$$
        is two, then $\mathbb F=\mathbb R$ and $\mathfrak A = \mathbb {\mathcal M}_n(\mathbb R)$.
        \item[(iv)] Otherwise, $\mathbb F=\mathbb C$ and $\mathfrak  A = \mathbb {\mathcal M}_n(\mathbb C)$.
    \end{itemize}
\end{theorem}

As shown in the example below  we must exclude the case of one-dimensional $C^\ast$-algebra in Theorem \ref{thm:simple}. 
\begin{example} The graphs $\boldsymbol{\Gamma}$ can not distinguish between one-dimensional real $C^*$-algebra $\mathbb {\mathcal M}_1(\mathbb R)$ and one-dimensional complex $C^*$-algebra $\mathbb {\mathcal M}_1(\mathbb C)$. Namely, in these cases,  scalars $x$ and $y$ satisfy  $x\perp y$ if and only if at least one of them is zero. Hence,
    \begin{itemize}
        \item[(a)] $\Gamma_0(\mathbb {\mathcal M}_1(\mathbb R))$ and $\Gamma_0(\mathbb {\mathcal M}_1(\mathbb C))$ are isomorphic to a star graph $K_{1,\infty}$ on continuum many vertices, where the center of the star has a loop.
        \item[(b)] $\Gamma(\mathbb {\mathcal M}_1(\mathbb R))$ and $\Gamma(\mathbb {\mathcal M}_1(\mathbb C))$ are isomorphic to  a null graph $ K_1$ on a single vertex (and no loops).
        \item[(c)] $\hat{\Gamma}_0(\mathbb {\mathcal M}_1(\mathbb R))$ and $\hat{\Gamma}_0(\mathbb {\mathcal M}_1(\mathbb C))$ are isomorphic to an edge $K_2$ with exactly one looped vertex. 
    \end{itemize}
\end{example} 

The above theorem implies that the relation of  BJ orthogonality alone can uniquely, up to  $\ast$-isomorphism, classify the finite-dimensional simple $C^\ast$-algebras of dimension greater than one. More precisely, we have the following. 

\begin{corollary}  Let  $\mathfrak A_1$ and $\mathfrak A_2$ be two simple $C^\ast$ algebras over $\mathbb F_1$ and $\mathbb F_2$, respectively. If $\mathfrak A_1$ is finite-dimensional with dimension greater than one, then the following are equivalent:
\begin{itemize}
\item[(i)] $\mathfrak A_1\sim_{\Gamma_0}\mathfrak A_2$,
\item[(ii)] $\mathfrak A_1\sim_{\Gamma}\mathfrak A_2$,
\item[(iii)] $\mathfrak A_1\sim_{\hat{\Gamma}_0}\mathfrak A_2$,
\item[(iv)] $\mathbb F_1=\mathbb F_2$ and $\mathfrak A_1$ and $\mathfrak A_2$ are isometrically $\ast$-isomorphic $C^\ast$-algebras.
\end{itemize} 
\end{corollary}
 

\section{Proofs}\label{section3}

Throughout this section, $\mathfrak A$ will be a finite-dimensional simple $C^*$-algebra. We write $\mathfrak A = \mathbb {\mathcal M}_n(\mathbb K)$ ($\KK\in\{\RR,\CC,\HH\}$) if the underlying field is real or not specified, and we write  $\mathfrak A=\mathbb {\mathcal M}_n(\mathbb C)$ if the underlying field is complex. 
Recall  that $\mathfrak A$ acts on $\mathbb K^n$ if $\FF=\RR$ and acts on $\mathbb C^n$ if $\FF=\CC$, by left multiplication of matrices and column vectors. Given $u, v\in\mathbb K^n$, we define the ${\mathbb R}$-linear functional  $F_{u, v}$ (when $\mathbb F=\mathbb R$)  on $\mathfrak A$ 
as $$F_{u,v}(B) = \text{Re}(v^*Bu),$$ and  we define the ${\mathbb C}$-linear functional, again denoted by $F_{u, v}$ (when $\FF=\CC$) on $\mathfrak A$ 
as $$F_{u,v}(B) = v^*Bu.$$ In both cases, its  operator norm equals $\|F_{u, v}\| = \|u\|\|v\|$. 
Also, by using matrix representation
$\left[\begin{smallmatrix}
    a & b\\
    -b &a
\end{smallmatrix}\right]$ and   $\left[\begin{smallmatrix}
    a+bi & c+di\\
    -c+di & a-bi
\end{smallmatrix}\right]$ for complex number $a+bi$ and quaternion $a+bi+cj+dk$, respectively, we can embed $\mathbb {\mathcal M}_n(\mathbb C)$ and $\mathbb {\mathcal M}_n(\mathbb H)$ (considered as a $C^*$-algebra over the field $\mathbb R$) $\ast$-isomorphically into $(\mathbb {\mathcal M}_{2n}(\mathbb R), \|\cdot\|)$ and $(\mathbb {\mathcal M}_{4n}(\mathbb R), \|\cdot\|)$ respectively. 

Each $A\in\mathbb {\mathcal M}_n(\mathbb K)$ has a singular value decomposition, that is, there exist $\KK$-orthonormal bases  $\{u_1, \dots, u_n\}$ and $\{v_1, \dots, v_n\}$  (i.e., $u_i^\ast u_j=v_i^\ast v_j=\delta_{ij}$, Kronecker delta) in $\KK^n$ such that $A=\sum\limits_{i=1}^n\sigma_iv_iu_i^*$, where $\sigma_1\geq\sigma_2\geq\dots\geq \sigma_n\ge0$ are the singular values of $A$
(see \cite[Theorem 7.2]{zhang} for the singular value decomposition of  quaternionic matrices).
Finally, for $A\in\mathbb {\mathcal M}_n(\mathbb K)$ we define the set of its norm-attaining vectors by  
 $$M_0(A) = \{u\in\mathbb K^n;\;\; \|Au\|=\|A\|\|u\|\}.$$
 It  is well-known that this is a $\KK$-vector subspace of $
 \KK^n$, at least  when  $\KK\in\{\RR,\CC\}$. 
 We nonetheless supply the proof, which will also cover the case of  $\mathbb K=\mathbb H$, for the sake of convenience.
 
\begin{lemma}\label{lem10} For $A\in\mathbb {\mathcal M}_n(\mathbb K)$, we have $M_0(A)=\mathrm{Ker}\,(A^\ast A-\|A\|^2\,I)$.
\end{lemma}
\begin{proof} Using the singular value decomposition, we have $A = \sum\limits_{i=1}^n \sigma_i v_iu_i^*$ where $\sigma_1\ge \dots\ge\sigma_{n}\ge0$. We claim that $M_0(A)$ is the $\mathbb K$-linear span of 
orthonormal vectors $\{u_1, \dots, u_k\}$ where $k$ is the largest integer such that $\sigma_k=\sigma_1$. Let $x\in\mathbb K^n$. Then $x=\sum\limits_{i=1}^n u_i\alpha_i$ for some $\alpha_i\in \mathbb K$, and hence 
$$Ax = \sigma_1\bigg(\sum\limits_{i=1}^k  v_i\alpha_i\bigg) +\sum\limits_{i=k+1}^n \sigma_i  v_i\alpha_i.$$ Since $\{v_1,\dots,v_n\}$ is an orthonormal basis, Pythagorean theorem gives $$\|Ax\|^2=\sigma_1^2\bigg(\sum\limits_{i=1}^k |\alpha_i|^2\bigg)+\sum\limits_{i=k+1}^n \sigma_i^2|\alpha_i|^2.$$ Since $\|A\|=\sigma_1$ and $\sigma_i<\sigma_1$,  ($i\ge k+1$) we have, $\|Ax\|^2=\|A\|^2\|x\|^2=\sigma_1^2\bigg(\sum\limits_{i=1}^n |\alpha_i|^2\bigg)$ if and only $\alpha_i=0$ for all $k+1\leq i\leq n$. This proves that the $\KK$-linear span of $\{u_1, \dots, u_k\}$ is equal to $M_0(A)$.

To finish the proof, we note that $A^*A = \sum\limits_{i=1}^n \sigma_i^2 u_iu_i^*$, so, $$A^*A -\|A\|^2I = \sum\limits_{i=k+1}^n (\sigma_i^2-\sigma_1^2)u_iu_i^*.$$ Hence, $x\in \mathrm{Ker}(A^*A -\|A\|^2I)$ if and only if $x$ is in the $\KK$-linear  span of $\{u_1, \dots, u_k\}$, which as we already know, is equal to $M_0(A)$.
\end{proof}

The next classification of BJ orthogonality in  $\mathfrak A$ is crucial for us. 

\begin{proposition}[Stampfli-Magajna-Bhatia-\v Semrl]\label{prop1} Let $A, B\in \mathfrak A$. Then, $B\in A^\bot$ if and only if there exist a unit vector $u\in M_0(A)$ such that $B\in \mathrm{Ker}(F_{u,\,Au})$.
\end{proposition}

\begin{proof}
If $\mathbb F=\mathbb C$ this was proven in \cite[Theorem 2]{STAMPF} for the case when $B=I$, and then in  \cite{MAGAJN} for a general~$B$. An alternative proof was given in \cite[Theorem 1]{bhatia} whose arguments work equally well for the case $\mathbb {\mathcal M}_n(\mathbb R)$ over the field $\FF=\mathbb R$. 
Finally, when considering $\mathbb {\mathcal M}_n(\mathbb C)$ and $\mathbb {\mathcal M}_n(\mathbb H)$ as $C^*$-algebras over $\mathbb F=\mathbb R$, we embed them 
into $\mathbb {\mathcal M}_{2n}(\mathbb R)$ and $\mathbb {\mathcal M}_{4n}(\mathbb R)$, respectively. 
\end{proof}
A generalization of Proposition \ref{prop1} was given in \cite[Corollary 1.3]{groversingla} and \cite[Corollary 2.5]{singla}. 


To simplify the next and further statements we will be using the same symbol $A$ for the vertices in $\boldsymbol{\Gamma}(\mathfrak A)$  as well as for (one of) its  representatives in $\mathfrak A$. That is,   instead of   `if $A\in\mathfrak A$ then $[A]\in\boldsymbol{\Gamma}(\mathfrak A)$'  we will write simply `if $A\in\mathfrak A$ then $A\in\boldsymbol{\Gamma}(\mathfrak A)$;'  the distinction will be clear from the context.

\begin{lemma}\label{lem1}
Let $A_1, A_2$ be two loopless vertices in $\boldsymbol{\Gamma}(\mathfrak A)$. Then $A_1^\bot\subseteq A_2^\bot$ if and only if $M_0(A_1)\subseteq M_0(A_2)$ and there exists a nonzero $\alpha\in \mathbb F$, the underlying field of $\mathfrak A$, such that $A_1u=\alpha(A_2u)$ for all $u\in M_0(A_1)$.
\end{lemma}
\begin{proof} By~\eqref{eq:connect} we can assume that $\boldsymbol{\Gamma} = \Gamma_0$. We provide the proof for the case when $\mathfrak A=\mathbb {\mathcal M}_n(\mathbb K)$, is a real $C^*$-algebra;  the proof when  $\mathfrak A=\mathbb {\mathcal M}_n(\mathbb C)$ is  a complex  $C^*$-algebra  follows along the similar lines --- we merely  replace $\text{Re}\,((A_iu)^\ast (Bu))=0$ with $(A_iu)^\ast (Bu)=0$ and  $\alpha\in \mathbb R$ with $\alpha\in \mathbb C$.

Recall that loopless vertices correspond to nonzero elements. Hence, by homogeneity  of BJ orthogonality, we can assume that $\|A_1\|=\|A_2\|=1$. Using Proposition \ref{prop1}, $B\in A_i^\bot$ if and only if 
there exists a unit vector $u\in M_0(A_i)$ such that $B\in \text{Ker}(F_{u,\,A_iu})$. Hence, \begin{equation}\label{eq2}A_i^\bot = \bigcup_{u\in M_0(A_i)} \text{Ker}(F_{u,A_iu}).\end{equation} 
Therefore, if $M_0(A_1)\subseteq M_0(A_2)$ and there exists  a nonzero $\alpha\in \mathbb R$  such that $A_1u=\alpha(A_2u)$ for all $u\in M_0(A_1)$, then using \eqref{eq2}, we get $A_1^\bot\subseteq A_2^\bot$. 

Conversely,  if  $A_1^\bot\subseteq A_2^\bot$, then for all unit vectors $u\in M_0(A_1)$, the hyperplane  $\text{Ker}(F_{u,A_1u})$
is contained in $A_2^\bot$. Using \cite[Theorem 2.1]{james},
(whose proof, though stated for real spaces,  works verbatim over real as well as  complex normed spaces)
it implies that $F_{u,A_1u}$ is a supporting functional of $A_2$, that is,
$$|F_{u,A_1u}(A_2)| = \|F_{u,A_1u}\|\cdot \|A_2\|=\|u\|\cdot\|A_1u\|\cdot\|A_2\|,$$ which, by the definition of $F_{u,A_1u}$  and as $u\in M_0(A_1)$ is a unit vector,  is equivalent to  \begin{equation}\label{eq1}\big|\text{Re}((A_1u)^*A_2u)\big| = \|A_1u\|\|A_2\|=\|A_1\|\cdot\|A_2\|=1\cdot 1=1.\end{equation} Using Cauchy-Schwarz inequality, we get $\|A_2u\| = 1$. It implies $u\in M_0(A_2)$ and therefore $M_0(A_1)\subseteq M_0(A_2)$. 

Using the condition of equality in Cauchy-Schwarz inequality in \eqref{eq1}, we further get  $A_1u = \alpha(A_2u)$ for  $\alpha\in \mathbb R$ with $|\alpha|=1$. Therefore, for every two unit vectors $u, v\in M_0(A_1)$ there exist scalars $\alpha, \beta$ with $|\alpha|=|\beta|=1$ such that $A_1u = \alpha(A_2u)$ and $A_1v = \beta(A_2v)$. It only remains to show that $\alpha = \beta$. 
We first observe that there exists a unit vector $w\in M_0(A_1)$ such that $$\dfrac{\alpha+\beta}{2}=\dfrac{\text{Re}((A_2u)^*A_1u)+\text{Re}((A_2v)^*A_1v)}{2}=\text{Re}((A_2w)^*A_1w);$$ in the last step, we have used the convexity of $\{\text{Re}(u^*A_1^*A_2u);\;\; \|u\|=1, u\in M_0(A_1)\}$, the numerical range of the compression of  $A_1^*A_2$   to the subspace $M_0(A_1)$ (recall that $A_1^*A_2$ is embedded into a suitable ${\mathcal M}_{n}(\RR), {\mathcal M}_{2n}(\RR)$ or ${\mathcal M}_{4n}(\RR)$). Again, $A_1^\bot\subseteq A_2^\bot$ implies that $A_1w = \gamma(A_2w)$ for some $\gamma$ with $|\gamma|=1$. Thus, $|\alpha+\beta|=2|\text{Re}((A_2w)^*(\gamma A_2w))|=2$. But $|\alpha|=|\beta|=1$ and hence  $\alpha=\beta$.
\end{proof}

\begin{corollary}\label{cor1} Let $A_1, A_2$ be two loopless vertices in $\boldsymbol{\Gamma}(\mathfrak A)$. Then, $A_1^\bot= A_2^\bot$ if and only if $M_0(A_1)=M_0(A_2)$ and there exists nonzero $\alpha\in \mathbb F$, the underlying field of $\mathfrak A$, such that $A_1u=\alpha(A_2u)$ for all $u\in M_0(A_1)$.
\end{corollary}
\begin{proof} This directly follows from the last lemma.\end{proof}

We proceed by proving a few properties about maximal chains in $\boldsymbol{\Gamma}(\mathfrak A)$ which are  required to prove 
Theorem \ref{thm:simple}. 

\begin{lemma}\label{thm1} Let $\mathfrak A={\mathcal M}_n(\KK)$ and let $(A_1,\dots, A_m)$ be a chain of loopless vertices in $\boldsymbol{\Gamma}(\mathfrak A)$. Then, $m\leq n$. Also, if $m=n$, then $\dim_{\KK}M_0(A_i)= i$ for all $1\leq i\leq n$.
\end{lemma}
\begin{proof} Let $1\leq j\leq m-1$ be fixed. Using Lemma \ref{lem1}, we have $M_0(A_j)\subseteq M_0(A_{j+1})$ and there exists a nonzero $\alpha_j \in \FF$, the underlying field of $\mathfrak A$, such that $A_ju=\alpha_j (A_{j+1}u)$ for all $u\in M_0(A_j)$. 

Assume $M_0(A_j)=M_0(A_{j+1})$, then using Corollary \ref{cor1}, we get $A_j^\bot=A_{j+1}^\bot$, which is not the case because $(A_1,\dots, A_m)$ is a chain. It implies $M_0(A_j)\subsetneq M_0(A_{j+1})$. Since, by Lemma \ref{lem10}, $M_0(A_j)$ and $M_0(A_{j+1})$ are subspaces, we get $\dim_{\mathbb K}M_0(A_{j+1})>\dim_{\mathbb K}M_0(A_{j})$. Combined with the fact that $A_j$  are loopless vertices, hence $A_j\neq0$, this further implies that $\dim_{\mathbb K}M_0(A_i)\geq i$ for all $1\leq i\leq m$. As such, $m\leq n$, and if $m=n$, then $\dim_{\mathbb K}M_0(A_i)=i$ for all $1\leq i\leq m$.
\end{proof}

\begin{corollary}\label{cor5} If $\mathfrak A={\mathcal M}_n(\KK)$, then every maximal chain of loopless vertices in $\boldsymbol{\Gamma}(\mathfrak A)$ is of  length~$n$.
\end{corollary}

The  length of  maximal chain in ${\bf\Gamma}(\mathfrak A)$ is clearly a graphological property. 
We show next that  being a scalar multiple of unitaries in $\mathfrak A$ is also a graphological property. Recall that a vertex  $x$ in a general digraph is called \textit{left-symmetric}  if $x\rightarrow y$  implies $y\rightarrow x$ for every vertex $y$. In case a diagraph corresponds to one among ortho-digraphs ${\bf\Gamma}(\mathfrak A)$, then a vertex $A$ is left-symmetric if $A\perp B$ implies $B\perp A$  for each (representative) $B\in\mathfrak A$. 
Similarly, a vertex  $x$ in a general diagraph is called \textit{right-symmetric}, if  $y\rightarrow  x$ implies $x\rightarrow  y$
(when a diagraph corresponds to ${\bf\Gamma}(\mathfrak A)$  this reduces to  $B\perp A$ implies $A\perp B$). 
We remark that, in connection with  BJ orthogonality, right symmetricity was  studied by Turn\v sek who obtained the equivalence (iii) $\Leftrightarrow$ (iv), in lemma below, for $B(H)$ with $H$  a complex Hilbert space (see~\cite[Theorem 2.5]{Turnsek}) while left/right symmetricity was explicitly studied  by Sain~\cite{Sain}, see also Tanaka~\cite{tanaka5}. 

\begin{lemma}\label{lem:unitary} Let $\mathfrak A={\mathcal M}_n(\KK)$.  The following are equivalent for a loopless vertex $A\in\boldsymbol{\Gamma}(\mathfrak{A})$:
\begin{itemize}
    \item[(i)]  $A$ is the last element of some maximal chain of loopless vertices in $\boldsymbol{\Gamma}(\mathfrak{A})$.
    \item[(ii)] Among loopless vertices, $A$ has  maximal outgoing neighborhood in set-theoretical sense.
    \item[(iii)] $A$ is a right-symmetric vertex in $\boldsymbol{\Gamma}(\mathfrak A)$.
\item[(iv)] Every  representative of $A$, when considered as an element in $\mathfrak{A}$, corresponds to a nonzero multiple of some unitary $U$.

\end{itemize}
Furthermore, there are no loopless left-symmetric vertices in $\boldsymbol{\Gamma}(\mathfrak A)$.
\end{lemma}

\begin{proof}  (i) $\Leftrightarrow$ (ii). Every matrix belongs to  some maximal chain and its last element has, by definition, the maximal outgoing neighborhood.

(i) $\Leftrightarrow$ (iv). By Lemma~\ref{thm1} every maximal chain ends up in a vertex which corresponds to a matrix $A$ with $\dim_{\KK} M_0(A)=n$, i.e., with $M_0(A)=\KK^n$. By singular value decomposition, such matrix is necessarily a scalar multiple of a unitary matrix. Conversely, let $A$ be  a scalar multiple of a unitary matrix, then $A$ has singular value decomposition $A=\sigma\sum\limits_{i=1}^n v_iu_i^*$. Then,  by Lemma~\ref{lem1} and its corollary, $\bigg(v_1u_1^*\,,\, \sum\limits_{i=1}^2 v_iu_i^*\,,\, \sum\limits_{i=1}^3 v_iu_i^* ,\dots, A\bigg)$ is the required maximal chain in $\boldsymbol{\Gamma}(\mathfrak A)$.

 (iii) $\Leftrightarrow$ (iv). First assume $A$ is right-symmetric. Recall that with unitary $U,V$, the map $X\mapsto UXV^\ast$ is an isometry, and hence induces an automorphism of the ortho-digraph. So we can assume with no loss of generality that $$A=\Sigma=\diag(s_1,\dots,s_n)$$ with $s_1=\dots=s_k>s_{k+1}\ge\dots\ge s_n\ge0$ for some $k\in\{1,\dots,n\}$. We claim that $k=n$, i.e., $A$ is a scalar multiple of unitary.


Assume, if possible, $k\le n-1$, denote $\sigma:=\frac{s_{k+1}}{s_k}=\frac{s_{k+1}}{s_1}\in[0,1)$, and consider 
$$B=e_1e_1^*+\dots+e_{k-1}e_{k-1}^*+u_{k}v_k^\ast+u_{k+1}v_{k+1}^\ast,$$  where  $\{e_j;\;\;1\le j\leq n\}$ denotes the standard basis of $\KK^n$ and
$$u_k=\frac{e_k-e_{k+1}}{\sqrt{2}},\quad v_k=\frac{\sigma e_k
   +e_{k+1}  }{\sqrt{1+\sigma^2}},\quad u_{k+1}=\frac{e_k+e_{k+1}}{\sqrt{2}},\quad v_{k+1}=\frac{e_k-\sigma e_{k+1}}{\sqrt{1+\sigma ^2}}.$$
   Notice that $B$ is already in its singular value decomposition and achieves its norm on $v_{k}$, which is mapped into $Bv_{k}=u_{k}$ while $Av_{k}=\frac{s_{k+1}}{\sqrt{1+\sigma^2}} (e_k+e_{k+1})$  is clearly orthogonal to $u_k$ relative to $\FF$-valued inner product on $\KK^n$
   $$\langle u,v\rangle_{\FF}:=\begin{cases}
       \text{Re}(v^*u); & \FF=\RR\\
       v^*u; & \FF=\CC
   \end{cases}.$$
   
Thus, $B\perp A$ using Proposition \ref{prop1}. However,   the compression of $B$ to  a subspace 
   $$M_0(A):=\{x\in \FF^n;\;\; 1=\|x\|=\|Ax\|=\|A\|\}=\Span\{e_1,\dots,e_{k}\}$$
   equals $\left(\begin{smallmatrix}
       I_{k-1} & 0 \\
 0 & \frac{1+\sigma}{\sqrt{2} \sqrt{1+\sigma ^2}} \\
   \end{smallmatrix}\right)$. Hence,  it is positive-definite and as such,
   $\langle Ax,Bx\rangle_{\mathbb F}=s_k \langle  x,Bx\rangle_{\mathbb F}>0$ for nonzero $x\in M_0(A)$ which gives $A\not\perp B$. Thus, $A$ cannot be right-symmetric.

Conversely, if $A$ is a scalar multiple of a unitary then it achieves its norm on every nonzero vector~$x$. Consider an arbitrary $B\perp A$; then $B$ achieves its norm on some vector $x$ with $\langle Bx,Ax\rangle_{\mathbb F}=0$. Since $A$ also achieves its norm on the same vector we see that $A\perp B$ also holds, so that $A$ is a right-symmetric vertex. 

Finally, we prove that there are no loopless left symmetric vertices. Assume otherwise and let   $A=U\Sigma V^*\in \mathfrak A$ be a non-zero left symmetric matrix written in its singular value decomposition. Since $X\mapsto U^\ast X V$ is 
an isomorphism of the digraph $\boldsymbol{\Gamma}(\mathfrak A)$, we can assume that  $A=\Sigma=\diag(s_1,\dots,s_n)$ is diagonal with $s_1\ge\dots\ge  s_n\ge 0$. Consider now a matrix $B=e_2e_2^\ast\in\mathfrak A$. Notice that $A$ attains its norm on $e_1$ and $\langle Ae_1,Be_1\rangle_{\FF}=0$, so $A\perp B$ using Proposition \ref{prop1}. Hence, by left-symmetricity, also $B\perp A$. Now, $B$ attains its norm only on (a multiple of)  $e_2$, and maps it into (a multiple of) $e_2$, so $0=\langle Be_2,Ae_2\rangle_{\mathbb F}=s_2$. Thus, $A=\diag(s_1,0,\dots,0)$.

Lastly, consider $B=(c e_1+se_2)(c e_1+se_2)^\ast-\tfrac{1}{3}(-s e_1+ce_2)(-s e_1+ce_2)^\ast$ where $(c,s)=(-1/2,\sqrt{3}/2)$. Notice that $A$ attains its norm on $e_1$ and that $\langle Ae_1,Be_1\rangle_{\mathbb F}=0$, so $A\perp B$. Notice also that $B$ attains its norm only on a multiple of $x=c e_1+se_2$ and that $\langle Bx,Ax\rangle_{\mathbb F}=s_1\,c^2\neq0$, so $B\not\perp A$. Hence, $A$ is not left-symmetric. 
\end{proof}

We say that a collection of  matrices from ${\mathcal M}_n(\mathbb K)$ has a simultaneous singular value decomposition if there exists unitary matrices $U$ and $V$ such that $UAV^\ast$ is  diagonal with real non-negative entries for every $A$ from the collection. We say that a collection of vertices from $\boldsymbol{\Gamma}(\mathbb {\mathcal M}_n(\mathbb K))$ has a simultaneous singular value decomposition if there exists representatives for each vertex in the collection with the simultaneous singular value decomposition. Recall that by our convention, we denote the vertices in $\boldsymbol{\Gamma}(\mathfrak A)$ and their  matrix representatives with the same symbols.


\begin{lemma}\label{lem7} Let $(A_1, \dots, A_n)$ be a maximal chain of loopless vertices in $\boldsymbol{\Gamma}(\mathfrak A)$. There exist $n$ matrices $B_1, \dots, B_{n-1}, A_n$ which have a simultaneous singular value decomposition such that $B_i^\bot=A_i^\bot$ for all $1\leq i\leq n-1$. Moreover, if $A_{n}$ is a diagonal matrix,
then $A_{n-1}$ is also a diagonal matrix.
\end{lemma}
\begin{proof} 
Without of loss of generality, $\|A_i\|=1$ for all $1\leq i\leq n$. By Lemma  \ref{thm1}, \mbox{$\dim_{\mathbb K}M_0(A_i)=i$}. Let $u_1$ be a normalized vector in $M_0(A_1)$. By Lemma \ref{lem1}, $M_0(A_1)\subseteq M_0(A_2)$, so $u_1\in M_0(A_2)$. Since $\dim_{\mathbb K}M_0(A_2)=2$, we choose $u_2\in M_0(A_2)$ such that $\{u_1, u_2\}$ is a $\mathbb K$-orthonormal basis of $M_0(A_2)$. By induction, there exist $\mathbb K$-orthonormal vectors $u_1,\dots, u_n$ such that for all $1\leq i\leq n$, we have $\{u_1,\dots, u_i\}$ is a $\mathbb K$-orthonormal basis of $M_0(A_i)$. By Lemma \ref{lem:unitary}, $A_n$ is a unitary matrix. So, $A_n$ has a singular value decomposition $A_n=\sum\limits_{k=1}^n v_ku_k^*$. Using Lemma \ref{lem1} again, there exists $\beta\in\FF\setminus\{0\}$  such that 
\begin{equation}\label{eq:maxchain}
    A_{n-1}u_k = \beta (A_{n}u_k)\quad \text{ for all }\   1\leq k\leq n-1,
\end{equation}
and  
\begin{equation}\label{eq:aux1}
    A_{n-1}^\bot=\bigg(\sum\limits_{k=1}^{n-1} v_{k} u_{k}^*\bigg)^\bot.
\end{equation}
 So, we can choose $B_{n-1}=\sum\limits_{k=1}^{n-1} v_k u_k^*$. Proceeding similarly, there exists $\gamma\neq0$ such that $A_{n-2}u_k = \gamma (A_{n-1}u_k)$ for all $1\leq k\leq n-2.$ So, we can choose $B_{n-2}=\sum\limits_{k=1}^{n-2} v_ku_k^*$. Along same lines, we can choose$$B_i=\sum\limits_{k=1}^{i} v_{k} u_{k}^*\quad \text{ for all }\ 1\leq i\leq n-1.$$

To prove the second statement, let  $A_{n-1}=\sum \sigma_ix_iy_i^\ast$ be its singular value decomposition in which $\sigma_1=\dots=\sigma_k>\sigma_{k+1}\ge\dots\ge\sigma_n\ge0$ for some $k$. Then, $M_0(A_{n-1})=\Span_{\KK}\{y_1,\dots,y_k\}$ which, by Lemma~\ref{thm1}, is an $(n-1)$-dimensional subspace of $\KK^n$, so $k=n-1$. Also, by  \eqref{eq:aux1} and Corollary~\ref{cor1}, 
\begin{equation}\label{eq:aux2}
    M_0(A_{n-1})=\Span\nolimits_{\KK}\{y_1,\dots,y_{n-1}\}=\Span\nolimits_{\KK}\{u_{1},\dots,u_{n-1}\}.
\end{equation} As such, modulo the $\KK$-scalar multiple, there exist unique $y_n,x_n$ which are $\KK$-orthogonal to \eqref{eq:aux2} and to $\Span_{\KK}\{x_1,\dots,x_{n-1}\}=\Span_{\KK}\{v_{1},\dots,v_{n-1}\}$ (see~\eqref{eq:maxchain}), respectively. It implies that $x_n=v_n\delta$ and $y_n=u_n\epsilon$ for some unimodular scalars $\delta,\epsilon\in\KK$ and as such, by~\ref{eq:maxchain}, $$A_{n-1}=\beta\sum_{i=1}^{n-1} v_iu_i^\ast+\sigma_n(v_{n}\delta)(\overline{\epsilon}u_{n}^\ast),$$
Thus, if $A_n$ is diagonal (i.e., if $v_i=u_i=e_i$ is a standard basis of $\KK^n$), then also  $A_{n-1}$ is.
\end{proof} 
Now, we give the proof of our main theorem.

\textit{Proof of Theorem \ref{thm:simple}}. The only finite-dimensional simple complex $C^*$-algebra is $\mathbb {\mathcal M}_n(\mathbb C)$ and the only finite-dimensional simple real $C^*$-algebras are $\mathbb {\mathcal M}_n(\mathbb R)$, $\mathbb {\mathcal M}_n(\mathbb C)$, and $\mathbb {\mathcal M}_n(\mathbb H)$.
Now, the dimension of $\mathbb {\mathcal M}_n(\mathbb C)$ over the field $\mathbb C$ is $n^2$, while the dimensions of $\mathbb {\mathcal M}_n(\mathbb R)$, $\mathbb {\mathcal M}_n(\mathbb C)$, and $\mathbb {\mathcal M}_n(\mathbb H)$ over the field $\mathbb R$ are $n^2$, $2n^2$ and $4n^2 (=(2n)^2)$, respectively. This  shows  (i) while  (ii) follows directly from Corollary \ref{cor5}. The only possibility left is when the dimension of $\boldsymbol{\Gamma}(\mathfrak A)$ is equal to $n^2$, i.e., to the  square of the length of some maximal chain in $\boldsymbol{\Gamma}(\mathfrak A)$. We hence have to distinguish between $(\mathfrak A,\FF)=({\mathcal M}_n(\RR),\RR)$ and  $(\mathfrak A,\FF)=({\mathcal M}_n(\CC),\CC)$ which will be  done by proving statements (iii) and (iv). By Lemma~\ref{lem7} we can assume without loss of generality that $(A_1,\dots, A_n)$ in $\boldsymbol{\Gamma}(\mathfrak A)$ is a maximal chain of (representatives of) nonzero diagonal matrices with eigenvalues $1$ or in modulus strictly smaller than one. Then, by Lemma~\ref{thm1}, $A_{n-1} = I_k\oplus\alpha\oplus I_{n-k-1}$ for some $0\leq k\leq n-1$ with $|\alpha|<1$ ($I_m$ denotes the $m\times m$ identity matrix). Using  Lemma~\ref{lem1} and equivalence between (i) and (iv) in Lemma~\ref{lem:unitary}, if a loopless $X\in\mathcal {\bf \Gamma}(A)$ satisfies $A_{n-1}^\bot\subsetneq X^\bot$, then, for a suitable nonzero $\gamma$,  $X= \gamma (I_k\oplus(\pm 1)\oplus I_{n-k-1})$ in case $\mathbb F=\mathbb R$ and $X= \gamma (I_k\oplus e^{it}\oplus I_{n-k-1})$ for some $t\in [0,2\pi)$ in case $\mathbb F=\mathbb C$. Clearly, with unimodular $\mu=e^{it}$ fixed, the vertices represented by $\mu (I_k\oplus \mu\oplus I_{n-k-1})$ share the same outgoing neighborhood for every nonzero $\gamma$. These neighborhoods are uniquely determined by $\mu$ because, by Corollary~\ref{cor1}, if $A=I_k\oplus\mu_1\oplus I_{n-k-1}$ and $B=I_k\oplus\mu_2\oplus I_{n-k-1}$ satisfy $A^\bot=B^\bot$, then there exists $\beta\in\FF$ with $A=\beta B$ 
which implies $\mu_1=\mu_2$. This proves the theorem. \qed

\section{Concluding remarks}\label{section4}

\begin{enumerate}
    \item[(A)] Note that in Lemma \ref{lem7}, we might not have the simultaneous singular value decomposition for  a maximal chain $(A_1, \dots, A_n)$ in  $\boldsymbol{\Gamma}(\mathfrak A)$. For example, 
$$\begin{bmatrix}1&0&0&0\\ 0&0&0&0\\0&0&0&0\\0&0&0&0\end{bmatrix}^\bot\subsetneq \begin{bmatrix}1&0&0&0\\ 0&1&0&0\\0&0&0&1/2\\0&0&1/2&0 \end{bmatrix}^\bot\subsetneq\begin{bmatrix}1&0&0&0\\ 0&1&0&0\\0&0&1&0\\0&0&0&0\end{bmatrix}^\bot\subsetneq\begin{bmatrix}1&0&0&0\\ 0&1&0&0\\0&0&1&0\\0&0&0&1\end{bmatrix}^\bot.$$ 

\item[(B)] The proof of Theorem \ref{thm:simple} also implies the following. In case of $\mathfrak A= \mathbb {\mathcal M}_n(\mathbb C)$, considered as a $C^*$-algebra over $\mathbb C$, the length of some (hence any) maximal chain in $\boldsymbol{\Gamma}(\mathfrak A)$ is $n$ and $\boldsymbol{\Gamma}(\mathfrak A)$ satisfies the property that for any maximal chain $(A_1,\dots, A_n)$ in $\boldsymbol{\Gamma}(\mathfrak A)$, the cardinality of $\{X\in \boldsymbol{\Gamma}(\mathfrak A);\;\; A_{n-1}^\bot\subseteq X^\bot\}$ is infinite.

\item[(C)]\label{remark:simple} Let $\mathfrak A$ be a finite dimensional complex $C^*$-algebra. Then, the following are equivalent: \begin{itemize} \item[(i)] $\mathfrak A$ is simple, \item[(ii)] the dimension of $\boldsymbol{\Gamma}(\mathfrak A)$ is a square of the length of the maximal chain.\end{itemize} 
\smallskip
To see this, embed $\mathfrak A= \mathbb {\mathcal M}_{n_1}(\mathbb C)\oplus\dots\oplus\mathbb {\mathcal M}_{n_\ell}(\mathbb C)$  into $\mathbb {\mathcal M}_N(\mathbb C)$, where $N=n_1+n_2+\dots+n_\ell$. As an application of Corollary \ref{cor5}, we get the length of maximal chain in $\mathfrak A$ is at most $N$. Let $e_1,\dots,e_N$ be standard unit vectors of $\mathbb C^n$. Let $P_k =\sum_{i=1}^k e_ie_i^*\in\mathfrak A$ for all $1\leq k\leq N$. Note that $P_k\in\mathfrak A$ and that the outgoing neighbourhoods of $P_k$ in $\mathfrak A$ increase with $k$ because  $P_{k+1}$  attains the norm at $e_{k+1}$ and so $P_k\in P_{k+1}^\perp\setminus P_k^\perp$. So, we get that the length of maximal chain in $\boldsymbol{\Gamma}(\mathfrak A)$ is $N$. The conclusion follows since the dimension of $\mathfrak A$ is $\sum_{j=1}^\ell n_j^2$, and $\sum_{j=1}^\ell n_j^2 = (n_1+n_2+\dots+n_\ell)^2$ if and only if $\ell=1$. 
 
\end{enumerate}

\end{document}